\documentclass[12pt]{article}%
\usepackage{amsmath}
\usepackage{amsfonts}
\usepackage{amssymb}
\usepackage{graphicx}%
\providecommand{\U}[1]{\protect\rule{.1in}{.1in}}
\newtheorem{theorem}{Theorem}

\newtheorem{corollary}[theorem]{Corollary}

\newtheorem{definition}[theorem]{Definition}

\newtheorem{lemma}[theorem]{Lemma}

\newtheorem{proposition}[theorem]{Proposition}
\newtheorem{remark}[theorem]{Remark}

\newenvironment{proof}[1][Proof]{\noindent\textbf{#1.} }{\ \rule{0.5em}{0.5em}}
\begin{document}

\title{Strong completeness of $L_{p}$-type vector lattices}
\author{Youssef Azouzi\\{\small Research Laboratory of Algebra, Topology, Arithmetic, and Order}\\{\small Department of Mathematics}\\{\small Faculty of Mathematical, Physical and Natural Sciences of Tunis}\\{\small Tunis-El Manar University, 2092-El Manar, Tunisia}}
\date{}
\maketitle

\begin{abstract}
Let $E$ be a Dedekind complete Riesz space with weak unit $e$, equipped with a
conditional expectation operator $T$. We prove that the spaces $L^{p}\left(
T\right)  ,$ $p\in\left[  1,\infty\right]  $, with their natural vector-valued
norms, are strongly complete, extending the $p=2$ case of Kuo, Kalauch, and
Watson. This resolves a question that has remained open for several years. We
begin by studying a general type of convergence and its unbounded
modification, unifying and generalizing order, norm, and absolute weak
convergence while providing simpler proofs. As an application, we consider
vector-valued norms and their unbounded variants, generalizing strong
convergence in $L^{p}$-spaces and convergence in probability. This framework
establishes the completeness of $L^{p}\left(  T\right)  $ and of the universal
completion $E^{u}$, reinforcing the uo-completeness of universally complete
vector lattices. Finally we apply our main theorem to obtain a new result in
ergodicity for conditional preserving systems.

\end{abstract}

\section{Introduction}

The paper adresses a very active topic in these last years, that is, unbounded
convergence. This is a way to associate to some type of convergence in vector
lattices or in Banach lattices its unbounded modification 'see
\cite{L-999,L-310}). We pay a particular attention to certain types of
convergence defiened on a Dedekind complete vector lattice $E$ with weak order
unit, equipped with a conditional expectation operator $T.$ The operator $T$
can be thought as an abtract, allowing to develop a measure-free theory of
stochastic analysis. The exploration of $L^{p}$-type spaces has been
previously undertaken in \cite{L-180}, where several basic aspects have been
developed. However, a crucial and fundamental property of $L^{p}\left(
T\right)  $ spaces, namely, strong completeness, defined in terms of its
appropriate vector-valued norm, has yet to be thoroughly explored in earlier
works. Strong completeness means here that every strong Cauchy net is strongly
convergent. If this property occurs only for sequences, it is referred to as
strong sequential completeness. Moreover in such circumstances, it appears
that these two notions are not equivalent. While important advancements have
been made previously, this paper aims to make significant contributions and
brings achievements to the subject. This work marks the fourth installment in
the exploration of this direction. The initial paper \cite{L-360}, authored by
Kuo, Rodda and Watson, proved the strong sequential completeness of $L%
{{}^1}%
(T)$ and the straightforward full completeness of $L^{\infty}(T)$ (for nets).
It is important to note that while the latter proof is relatively easy, the
former relies on a clever and distinct idea, diverging significantly from the
classical approach. Traditional proofs for the classical case are not directly
transferable to establish the sequential completeness of $L%
{{}^1}%
(T)$ in the general setting of Riesz spaces. The second paper \cite{L-1111},
authored by the present writer, extends the scope to sequential completeness
for other values of $p$. The proof is based on an earlier result established
in \cite{L-360} and the concept of $T$-uniformity -- a notion that generalizes
uniform integrability to the setting of Riesz spaces. In the third step,
Kalauch, Kuo and Watson addressed the challenge of establishing strong
completeness of $L^{2}\left(  T\right)  $, employing a new technique based on
a Hahn-Jordan type Theorem from a separate work (see \cite{L-886}). Given the
absence of a direct extension of the proof from the classical case, innovative
methods are crucial for achieving the desired outcome. This current
contribution aims to complete the picture, focusing on proving the full
completeness of other cases. This achievement will be derived from the strong
completeness of $L%
{{}^2}%
(T)$ established in \cite{L-928} and a theorem that we proved asserting that,
for any $p\in(1,\infty),$ the strong completeness of $L^{p}(T)$ is equivalent
to the strong completeness of $L%
{{}^1}%
(T)$. Following an idea of the author used to show that any universal complete
vector lattice is unbounded order complete, we are able here to prove that
$L^{1}\left(  T\right)  $ is complete with respect to convergence in
$T$-conditional probability, that is a reasonable generalization of
convergence in probability from the classical theory. We apply our theorem to
prove a result in ergodicity for conditional preserving system.

The paper is organized as follows. Section 2 provides some notations and
premilinaries. Section 3 investigates general unbounded convergences,
presenting basic facts, refining some earlier results, and providing shorter
proofs. This section introduces key ideas that will be essential in Sections 4
and 5, which present the main results of the paper on completeness. In Section
4 we prove the strong completeness of $L^{p}\left(  T\right)  $ for all
$p\in\left[  1,\infty\right]  .$ In Section 5, we prove that the universal
completion $E^{u}$ of $E$ is complete with respect to convergence in
$T$-conditional probability. Our results are then applied to conditional
expectation--preserving systems, leading in particular to a characterization
of ergodicity. For recent developments on this topic, we refer the reader to
\cite{L-950}, \cite{L-1150} and \cite{L-1100}.

\section{Some preliminaries}

We begin by fixing notation and reviewing the key concepts used throughout the
paper. After giving a brief discussion on nets and subnets, we recall the
relevant types of convergence, with particular emphasis on those defined via
conditional expectation operators.

\textbf{Nets and subnets. }A net in a set $X$ is a map $x$ from a directed
preordered $A$ into $X.$ It is usually denoted by $\left(  x_{\alpha}\right)
_{\alpha\in A}.$ It is worth noting that there are several nonequivalent
definitions of subnets. Perhaps the most useful, and the one we adopt here
(see \cite{b-2800}), is the following: Let $\left(  x_{\alpha}\right)
_{\alpha\in A}$ be a net in a set $X.$ A subnet of $\left(  x_{\alpha}\right)
_{\alpha\in A}$ is a net $\left(  y_{\beta}\right)  _{\beta\in B}$ such that
$y_{\beta}=x_{\varphi\left(  \beta\right)  }$ for all $\beta\in B,$ where
$\varphi:B\longrightarrow A$ is a cofinal map--that is, for every $\alpha
_{0}\in A$ there is $\beta_{0}\in B$ such that $\varphi\left(  \beta\right)
\geq\alpha_{0}$ for all $\beta\geq\beta_{0}.$ It is important to note that the
map $\varphi$ is not assumed to be increasing.

\textbf{Different types of convergences. }We start be recalling the most
fundamental type of convergence studied in vector lattices. A net $\left(
x_{\alpha}\right)  _{\alpha\in A}$ in a a vector lattice $X$ is said to be
order convergent to $x$ if there exists a net $\left(  y_{\gamma}\right)
_{\gamma\in\Gamma}$ such that $y_{\gamma}\downarrow0$ and for each $\gamma$ in
$\Gamma$ there exists $\alpha_{0}\in A$ such that $\left\vert x_{\alpha
}-x\right\vert \leq y_{\gamma}$ for every $\alpha\geq\alpha_{0}.$ The net
$\left(  x_{\alpha}\right)  $ is said to be unbounded order convergent to $x$
if for every $u\in X_{+}$ we have $\left\vert x_{\alpha}-x\right\vert \wedge
u\overset{o}{\longrightarrow}0.$ Unbounded norm convergence is defined for
nets in a normed vector lattice in a similar way where the latter condition is
replaced by: $\left\vert x_{\alpha}-x\right\vert \wedge u\overset{\left\Vert
.\right\Vert }{\longrightarrow}0$. Two other essential types of convergence
will be investigated in the following sections. Before recalling their
definitions, we need to introduce some terminology built around the
fundamental concept of conditional expectation operators.

\textbf{Condition Riesz triple. }We say that $\left(  E,e,T\right)  $ is a
conditional Riesz triple if $E$ is a Dedeklind complete vector lattice, $e$ is
a weak order unit of $E$ and $T$ is a conditional expectation operator, that
is a prositive order continuous projection with $R\left(  T\right)  $ a
Dedekind complete Riesz subspace. We assume, moreover, that $T$ is strictly
positive, that is, $T\left\vert x\right\vert =0$ if and only if $x=0.$ This
allows us to define a vector-valued norm on $E$ by setting%
\[
\left\Vert x\right\Vert _{T,1}=T\left\vert x\right\vert ,\qquad x\in E.
\]
It is shown in \cite{L-24} that there exists a largest Riesz subspace of
$E^{u}$, known as the \textbf{natural domain} of $T,$ denoted by $L^{1}(T),$
to which $T$ extends uniquely to a conditional expectation. One of the major
properties of the space $L^{1}\left(  T\right)  $ lies in its $T$-universal
completeness, which means that if $\left(  x_{\alpha}\right)  $ is an
increasing net in $L^{1}\left(  T\right)  _{+}$ with supremum $x$ in $E^{u}$
such that $\sup Tx_{\alpha}\in L^{1}\left(  T\right)  ^{u},$ then $x\in
L^{1}\left(  T\right)  .$ Functional calculus are used in \cite{L-180} to
define $L^{p}\left(  T\right)  $-spaces for $p\in\left(  1,\infty\right)  $,
namely%
\[
L^{p}\left(  T\right)  =\left\{  x\in E^{u}:\left\vert x\right\vert ^{p}\in
L^{1}\left(  T\right)  \right\}  .
\]
These spaces are equipped with vector-valued norm given by%
\[
\left\Vert x\right\Vert _{T,p}=\left(  T\left\vert x\right\vert ^{p}\right)
^{1/p}.
\]
Generalization of H\"{o}lder Inequality, Minkowski Inequality, and Lypunov
Inequality have been obtained also (see \cite{L-180}). Next, we give two more
types of convergences that will be of great interest in our current work.

\begin{definition}
We say that a net $\left(  x_{\alpha}\right)  _{\alpha\in A}$ is strongly
convergent to $x$ in $L^{p}\left(  T\right)  $ if $\left\Vert x_{\alpha
}-x\right\Vert _{T,p}\overset{o}{\longrightarrow}0.$
\end{definition}

The convergence in $T$-conditional probability has been introduced in
\cite{L-14} and it is an abstraction of convergence in probability.

\begin{definition}
Let $\left(  x_{\alpha}\right)  $ be a net in $E$ and $x\in E.$ We say that
$\left(  x_{\alpha}\right)  $ converges to $x$ in $T$-conditional probability
if $TP_{\left(  \left\vert x_{\alpha}-x\right\vert -\varepsilon e\right)
^{+}}e\overset{o}{\longrightarrow}0.$
\end{definition}

We recall next a Lemma proved in \cite[Lemma 4]{L-900}. Its statement
signifies also that $T$-conditional probability convergence aligns precisely
with unbounded norm convergence relative to the norm $\left\Vert .\right\Vert
_{1,T}$ within the $R(T)$-vector space.

\begin{lemma}
\label{UI-B}\textit{Let} $\left(  E,e,T\right)  $ \textit{be a conditional
Riesz triple}. \textit{For a net} $(x_{\alpha})_{\alpha\in A}$ \textit{in}
$E$, \textit{the following statements are equivalent}.

\begin{enumerate}
\item[(i)] $x_{\alpha}\rightarrow x$ \textit{in }$T$\textit{-conditional
probability};

\item[(ii)] $T(|x_{\alpha}-x|\wedge u)\overset{o}{\longrightarrow}$
\textit{for every} $u\in E_{+};$

\item[(iii)] $T(|x_{\alpha}-x|\wedge e)\overset{o}{\longrightarrow}0$.
\end{enumerate}
\end{lemma}

As convergence in probability theory is defined for non necessarily integrable
random variables, it is appropriate to extend the definition of convergence in
$T$-conditional probability to elements of $E^{u},.$the universal completion
of $E.$ Note also that Lemma \ref{UI-B} holds true for nets in $E^{u}$ without
any modification to its proof.

\section{General unbounded convergence}

We consider in this section some type of convergence $\theta,$ then consider
its unbounded modification and prove how several fundamental results can be
derived in unified and simplified way. Related approaches have been
investigated by other authors (see for example \cite{L-999}).

We consider throughout this section an Archimedean vector lattice $X$ equipped
with a convergence $\theta$ satisfying the following conditions:

\begin{enumerate}
\item[(1)] If $x_{\alpha}=x$ for every $\alpha$ then $x_{\alpha}%
\longrightarrow x;$

\item[(2)] If $x_{a}\overset{\theta}{\longrightarrow}x~$and $y_{\alpha
}\overset{\theta}{\longrightarrow}y$ then $x_{\alpha}+\lambda y_{\alpha
}\overset{\theta}{\longrightarrow}x+\lambda y$ and $tx\overset{\theta
}{\longrightarrow}0$ as $t\longrightarrow0;$

\item[(3)] If $x_{\alpha}\overset{\theta}{\longrightarrow}0$ and $\left\vert
y_{\alpha}\right\vert \leq\left\vert x_{\alpha}\right\vert $ then $y_{\alpha
}\overset{\theta}{\longrightarrow}0;$

\item[(4)] If $x_{\alpha}\overset{\theta}{\longrightarrow}x$ then $y_{\beta
}\overset{\theta}{\longrightarrow}x$ for every subnet $\left(  y_{\beta
}\right)  $ of $\left(  x_{\alpha}\right)  .$
\end{enumerate}

Convergences fulfilling the above conditions include topological convergence
within a locally convex-solid topology, order convergence and more generally
$p$-convergence for vector-valued norm $p.$ Convergences studied here exhibit
compatibility with the underlying order structure.

\begin{lemma}
\label{L0}Let $X$ be an Archimedean vector lattice equipped with a convergence
$\theta$ and consider two nets $\left(  x_{\alpha}\right)  _{\alpha\in A}$ and
$\left(  y_{\alpha}\right)  _{\alpha\in A}$ in $X.$ If $x_{\alpha}%
\overset{\theta}{\longrightarrow}x$ and $y_{\alpha}\overset{\theta
}{\longrightarrow}y$ then $x_{\alpha}\vee y_{\alpha}\overset{\theta
}{\longrightarrow}x\vee y$ and $x_{\alpha}\wedge y_{\alpha}\overset{\theta
}{\longrightarrow}x\wedge y.$
\end{lemma}

\begin{proof}
This can be deduced easily from the following two facts:

(i ) $\left\vert x_{\alpha}\vee y_{\alpha}-x\vee y\right\vert \leq\left\vert
x_{\alpha}-x\right\vert +\left\vert y_{\alpha}-y\right\vert $

(ii) $x_{\alpha}\wedge y_{\alpha}=x_{\alpha}+y_{\alpha}-x_{\alpha}\vee
y_{\alpha}.$
\end{proof}

Given a subset $Y$ of $X,$ the $\theta$-closure of $Y$ in $X,$ denoted by
$\overline{Y}^{\theta},$ is the set of all $x\in X$ such that there is a net
$\left(  x_{\alpha}\right)  _{\alpha\in A}$ in $Y$ such that $x_{\alpha
}\overset{\theta}{\longrightarrow}x.$ We say that $Y$ is $\theta$-closed if
$Y=\overline{Y}^{\theta}.$ We say that $Y$ is $\theta$-dense if $X=\overline
{Y}^{\theta}.$ In the case where $\theta$ is the order convergence, it is
typical for a sublattice $Y$ of $X$ to be order dense if for every $x>0$ there
is $y\in Y$ such that $0<y\leq x.$ Our definition, however, deviates from this
conventional definition and aligns with what is refereed to as density with
respect to order convergence (see \cite{L-65}). It was shwon that every order
dense sublattice is dense with respect to order convergence, but the converse
is not necessarily true (see \cite[Example 2.6]{L-65}).

Given any type of convergence $\theta$ satisfying conditions (1)-(4) above, we
can define the unbounded $\theta$-convergence (denoted $u\theta$-convergence)
as follows.

\begin{definition}
\label{Y4-A}Let $\left(  X,\theta\right)  $ be as above. A net $\left(
x_{\alpha}\right)  _{\alpha\in A}$ in $X$ is $u\theta$-convergent to $x$ if%
\begin{equation}
\left(  x_{\alpha}\vee b\right)  \wedge c\overset{\theta}{\longrightarrow
}\left(  x\vee b\right)  \wedge c, \label{C2}%
\end{equation}
for all $b,c\in X$ with $b\leq c.$
\end{definition}

\begin{remark}
\label{Y4-R}Note that if $a,b,c\in X$ with $b\leq c$ then we have%
\begin{align*}
\left(  a\vee b\right)  \wedge c  &  =\left(  a\wedge c\right)  \vee\left(
b\wedge c\right)  =\left(  a\wedge c\right)  \vee b;\\
\left(  \left(  -a\right)  \vee b\right)  \wedge c  &  =-\left[  \left(
a\vee\left(  -c\right)  \right)  \wedge\left(  -b\right)  \right]  .
\end{align*}
Thus $x_{\alpha}\overset{u\theta}{\longrightarrow}x$ if and only if $\left(
x_{\alpha}\wedge c\right)  \vee b\overset{\theta}{\longrightarrow}\left(
x\wedge c\right)  \vee b$ for all $b,c\in X$ with $b\leq c$ and that
$x_{\alpha}\overset{u\theta}{\longrightarrow}x$ if and only if $-x_{\alpha
}\overset{u\theta}{\longrightarrow}-x.$
\end{remark}

The next lemma, while straightforward, proves highly practical in application.

\begin{lemma}
\label{L1}The net $\left(  x_{\alpha}\right)  _{\alpha\in A}$ is $u\theta
$-convergent to $x.$ if and only if condition (\ref{C2}) holds for every
$b\leq0\leq c.$ If, in addition, $\left(  x_{\alpha}\right)  _{\alpha\in A}$
is a positive net then the $u\theta$-convergence is reduced to the following%
\begin{equation}
x_{\alpha}\wedge a\overset{\theta}{\longrightarrow}x\wedge a\text{ for all
}a\in X_{+}. \label{C3}%
\end{equation}

\end{lemma}

\begin{proof}
Assume that Condition (C2) holds for every $b\leq0\leq c.$ Now let $b,c\in X$
with $b\leq c.$ Then%
\begin{align*}
\left(  x_{\alpha}\vee b\right)  \wedge c  &  =\left(  x_{\alpha}\vee
b\vee\left(  -\left\vert b\right\vert \right)  \right)  \wedge c\wedge
\left\vert c\right\vert \\
&  =\left(  \left[  \left(  x_{\alpha}\vee\left(  -\left\vert b\right\vert
\right)  \right)  \wedge\left\vert c\right\vert \right]  \vee b\right)  \wedge
c,
\end{align*}
and it follows from the assumption and Lemma \ref{L0} that%
\[
\left(  x_{\alpha}\vee b\right)  \wedge c\overset{\theta}{\longrightarrow
}\left(  \left[  \left(  x\vee\left(  -\left\vert b\right\vert \right)
\right)  \wedge\left\vert c\right\vert \right]  \vee b\right)  \wedge
c=\left(  x\vee b\right)  \wedge c.
\]
The case of positive net $\left(  x_{\alpha}\right)  _{\alpha\in A}$ is now
obvious because we have%
\[
\left(  x_{\alpha}\vee b\right)  \wedge c=x_{\alpha}\wedge c
\]
for every $b\leq0\leq c\in X.$
\end{proof}

\begin{theorem}
\label{T1}Let $\left(  X,\theta\right)  $ be as above and let $\left(
x_{\alpha}\right)  _{\alpha\in A}$ be a net in $X.$ Then the following hold.

\begin{enumerate}
\item[(i)] If $x_{\alpha}\overset{u\theta}{\longrightarrow}x$ and $y_{\alpha
}\overset{u\theta}{\longrightarrow}y$ then $x_{\alpha}\vee y_{\alpha}%
\overset{u\theta}{\longrightarrow}x\vee y$ and $x_{\alpha}\wedge y_{\alpha
}\overset{u\theta}{\longrightarrow}x_{\alpha}\wedge y_{\alpha}$

\item[(ii)] If $x_{\alpha}\overset{u\theta}{\longrightarrow}x$ and $y_{\alpha
}\overset{u\theta}{\longrightarrow}y$ and $\lambda\in\mathbb{R}$ then $\lambda
x_{\alpha}+y_{\alpha}\overset{u\theta}{\longrightarrow}\lambda x+y,$

\item[(iii)] We have $x_{\alpha}\overset{u\theta}{\longrightarrow
}x\Leftrightarrow x_{\alpha}-x\overset{u\theta}{\longrightarrow}%
0\Leftrightarrow\left\vert x_{\alpha}-x\right\vert \overset{u\theta
}{\longrightarrow}0.$

\item[(iv)] $x_{\alpha}\overset{u\theta}{\longrightarrow}x\Leftrightarrow
\left\vert x_{\alpha}-x\right\vert \wedge u\overset{\theta}{\longrightarrow}0$
for all $u\in X_{+}.$
\end{enumerate}
\end{theorem}

The property given in (iv) is used in \cite{L-197}, \cite{L-173} and
\cite{L-354} to define uo-convergence, un-convergence and uaw-convergence
respectively. These convergences exihibit numerous shared properties, amenable
to a unified proof. Moreover, as the characterization of $u\theta$-convergence
given in Lemma \ref{L1}, specially for positive nets can be used to simplify
several proofs of known results. We will illustrate this through examples.
Even more, this characterization can be used to improve some results by
relaxing their assumptions (see for example Theorem \ref{T12} below).

\begin{proof}
(i) Assume that $x_{\alpha}\overset{u\theta}{\longrightarrow}x$ and
$y_{\alpha}\overset{u\theta}{\longrightarrow}y$ then for every $b\leq c$ we
have%
\begin{align*}
\left(  \left(  x_{\alpha}\vee y_{\beta}\right)  \vee b\right)  \wedge c &
=\left(  \left(  x_{\alpha}\vee b\right)  \wedge c\right)  \vee\left(
y_{\alpha}\vee b\wedge c\right)  \\
&  \overset{\theta}{\longrightarrow}\left(  \left(  x\vee b\right)  \wedge
c\right)  \vee\left(  y\vee b\wedge c\right)  \\
&  =\left(  \left(  \left(  x\vee y\right)  \vee b\right)  \wedge c\right)  ,
\end{align*}
which shows that $x_{\alpha}\vee y_{\alpha}\overset{u\theta}{\longrightarrow
}x\vee y.$ Similarly we show that $\left(  x_{\alpha}\wedge y_{\alpha}\right)
\overset{u\theta}{\longrightarrow}x\wedge y.$ Thus (i) is proved.

Assume that $x_{a}\overset{u\theta}{\longrightarrow}0.$ It follows from (i)
and Remark \ref{Y4-R} that $x_{\alpha}^{+}=x_{\alpha}\vee0\overset{u\theta
}{\longrightarrow}0$ and $x_{\alpha}^{-}=\left(  -x_{\alpha}\right)
\vee0\overset{u\theta}{\longrightarrow}0.$ Hence for every $u\in X_{+},$
\[
\left\vert x_{\alpha}\right\vert \wedge u=x_{\alpha}^{+}\wedge u+x_{\alpha
}^{-}\wedge u\overset{\theta}{\longrightarrow}0.
\]
Thus $\left\vert x_{\alpha}\right\vert \overset{u\theta}{\longrightarrow}0$ by
Lemma \ref{L1}.

Now if $x_{\alpha}\overset{u\theta}{\longrightarrow}x$ and $y\in X$ then
$x_{\alpha}+y\overset{u\theta}{\longrightarrow}x+y.$ This can be deduce from
the following identity:%
\[
\left(  \left(  x_{\alpha}+y\right)  \vee c\right)  \wedge b=\left(  \left(
x_{\alpha}\right)  \vee\left(  b-y\right)  \right)  \wedge\left(  c-y\right)
+y.
\]
In view of the preceding observation and using again Lemma \ref{L1} we get%
\[
x_{\alpha}\overset{u\theta}{\longrightarrow}x\Longleftrightarrow x_{\alpha
}-x\overset{u\theta}{\longrightarrow}0\Longleftrightarrow\left\vert x_{\alpha
}-x\right\vert \overset{u\theta}{\longrightarrow}0\Longleftrightarrow
\left\vert x_{\alpha}-x\right\vert \wedge u\overset{\theta}{\longrightarrow}0
\]
for all $u\in X_{+}.$It remains only to prove (ii). Assume that $x_{\alpha
}\overset{u\theta}{\longrightarrow}x$ and $y_{\alpha}\overset{u\theta
}{\longrightarrow}y$ let $\lambda\in\mathbb{R}$ and let $u\in X_{+}.$ Then
\[
\left\vert x_{\alpha}+\lambda y_{\alpha}-x-\lambda y\right\vert \wedge
u\leq\left\vert x_{\alpha}-x\right\vert \wedge u+\left\vert \lambda\right\vert
\left\vert y_{\alpha}-y\right\vert \wedge u\overset{\theta}{\longrightarrow
}0.
\]
Henc $x_{\alpha}+\lambda y_{\alpha}\overset{u\theta}{\longrightarrow}x+\lambda
y.$ This completes the proof.
\end{proof}

\begin{remark}
\label{R1}As it was noted in Lemma \ref{L1}, $u\theta$-convergence can be
reduced for positive nets $\left(  x_{\alpha}\right)  $ to the fact that
$x_{\alpha}\wedge u\overset{\theta}{\longrightarrow}x\wedge u$ for every $u.$
This is also valid for increasing nets: if $\left(  x_{\alpha}\right)  $ is an
increasing net then $x_{\alpha}\overset{u\theta}{\longrightarrow}x$ if and
only if $x_{\alpha}\wedge u\overset{\theta}{\longrightarrow}x\wedge u$ for all
$u\geq0.$ Indeed, if $x_{\alpha}\overset{u\theta}{\longrightarrow}x$ then for
$\alpha\geq\alpha_{0},$ and $b=-\left\vert u\right\vert -\left\vert
x_{\alpha_{0}}\right\vert -\left\vert x\right\vert $ we have%
\[
x_{\alpha}\wedge u=\left(  x_{\alpha}\vee b\right)  \wedge u\overset{u\theta
}{\longrightarrow}\left(  x\vee b\right)  \wedge u=x\wedge u.
\]
Conversely, if $x_{\alpha}\wedge u\overset{\theta}{\longrightarrow}x\wedge u$
for every $u$ then%
\[
\left(  x_{\alpha}-x_{\alpha_{0}}\right)  \wedge\left(  u-x_{\alpha_{0}%
}\right)  =x_{\alpha}\wedge u-x_{\alpha_{0}}\overset{\theta}{\longrightarrow
}x\wedge u-x_{\alpha_{0}}=\left(  x_{\alpha}-x_{\alpha_{0}}\right)
\wedge\left(  u-x_{\alpha_{0}}\right)  ,
\]
which implies that $\left(  x_{\alpha}-x_{\alpha_{0}}\right)  _{\alpha
\geq\alpha_{0}}$ $u\theta$-converges to $x-x_{\alpha_{0}}$ by Lemma \ref{L1}
and then $x_{\alpha}\overset{u\theta}{\longrightarrow}x$ by Theorem \ref{T1}.
\end{remark}

The following Lemma can be viewed as a generalization of \cite[Lemma
1.6]{L-421} and \cite[Lemma 1.2]{L-171}.

\begin{lemma}
Let $X$ be vector lattice and let $\theta$ be a convergence in $X$ for which
the limit is unique.

\begin{enumerate}
\item[(i)] If $x_{\alpha}\uparrow$ and $x_{\alpha}\overset{\theta
}{\longrightarrow}x$ then $x_{\alpha}\uparrow x.$

\item[(ii)] If $x_{\alpha}\uparrow$ and $x_{\alpha}\overset{u\theta
}{\longrightarrow}x$ then $x_{\alpha}\uparrow x$ and then $x_{\alpha}%
\overset{\theta}{\longrightarrow}x.$
\end{enumerate}
\end{lemma}

\begin{proof}
(i) Let $\beta$ be a fixed element in $A$ and consider the net $\left(
x_{\alpha}\wedge x_{\beta}\right)  _{\alpha\in A}.$ By Remark \ref{R1} we have
$x_{\alpha}\wedge x_{\beta}\overset{\theta}{\longrightarrow}x\wedge x_{\beta}$
and as $x_{\alpha}\wedge x_{\beta}=x_{\beta}$ for all $\alpha\geq\beta$ we
deduce that $x\geq x_{\beta}.$ Now if $z$ is an upper bound of $\left\{
x_{\alpha}\right\}  $ then $x_{\alpha}=x_{\alpha}\wedge z\overset{\theta
}{\longrightarrow}x\wedge z$ and then $x=x\wedge z\leq z.$

(ii) Using Remark \ref{R1} the proof is similar to (i).
\end{proof}

We say that $e$ is a $\theta$-unit if $x\wedge ne\overset{\theta
}{\longrightarrow}x$ for every $x\in X_{+}.$ When $\theta$ represents order
convergence, $e$ becomes a weak order unit, and when $\theta$ represents the
norm convergence in a Banach lattice, then $e$ serves as a quasi-interior
point, also known as a topological unit.

In the sequel, we assume that $\theta$ adheres to the following property:%

\begin{align}
\text{If }0  &  \leq x_{\alpha}\leq u_{\beta}+a\left(  \beta,\alpha\right)
\text{ with }u_{\beta}\overset{\theta}{\longrightarrow}0\text{ and }%
\tag{(*)}\\
&  \text{for each }\beta\text{, }a\left(  \beta,\alpha\right)  \underset
{\alpha}{\longrightarrow}0\text{ then }x_{\alpha}\overset{\theta
}{\longrightarrow}0.\nonumber
\end{align}

Property $\left(  \ast\right)  $ is readily satisfied for order convergence,
norm convergence in normed Riesz spaces, and strong convergence in
$L^{p}\left(  T\right)  .$

\begin{theorem}
Let $X$ be a vector lattice equipped with a convergence $\theta$ satisfying
condition $\left(  \ast\right)  .$ The following properties are equivalent.

\begin{enumerate}
\item[(i)] $e$ is a $\theta$-unit;

\item[(ii)] for every net $\left(  x_{\alpha}\right)  _{\alpha\in A}$ in $X$
we have $x_{\alpha}\overset{u\theta}{\longrightarrow}0$ if and only if
$\left\vert x_{\alpha}\right\vert \wedge e\overset{\theta}{\longrightarrow}0;$

\item[(iii)] for every sequence $\left(  x_{n}\right)  _{n\in\mathbb{N}}$ in
$X$ we have $x_{n}\overset{u\theta}{\longrightarrow}0$ if and only if
$\left\vert x_{n}\right\vert \wedge e\overset{\theta}{\longrightarrow}0.$
\end{enumerate}
\end{theorem}

\begin{proof}
(i) $\Longrightarrow$ (ii) Let $\left(  x_{\alpha}\right)  _{\alpha\in A}$ be
a net in $X_{+}$ such that $x_{\alpha}\wedge e\overset{\theta}{\longrightarrow
}0$ and let $u\in X_{+}.$ Then
\begin{align*}
x_{\alpha}\wedge u  &  =x_{\alpha}\wedge u-x_{\alpha}\wedge u\wedge
ne+x_{\alpha}\wedge u\wedge ne\\
&  \leq u-u\wedge ne+x_{\alpha}\wedge ne.
\end{align*}
As $\theta$ satisfies condition $\left(  \ast\right)  $ we conclude that
$x_{\alpha}\wedge u\overset{\theta}{\longrightarrow}0$.

(ii) $\Longrightarrow$ (iii) is trivial.

(iii) $\Longrightarrow$ (i). Let $x$ be an element in $X_{+}.$ We have to show
that $x_{n}:=x-x\wedge ne\overset{u\mathfrak{\theta}}{\longrightarrow}0.$
According to (iii) it is enough to show that $x_{n}\wedge e\overset{\theta
}{\longrightarrow}0$. Denote $u=x-ne$ and observe that%
\begin{align}
nx_{n}  &  =nu^{+}\wedge ne=nu^{+}\wedge\left(  x-u\right) \nonumber\\
&  =-u^{+}+\left(  n+1\right)  u^{+}\wedge x+u^{-}\theta\nonumber\\
&  =-u^{+}+\left(  n+1\right)  u^{+}\wedge x\leq x.
\end{align}
Thus, we have shown that $0\leq x_{n}\leq n^{-1}x,$ which implies that
$x_{n}\overset{\theta}{\longrightarrow}0$ and completes the proof.
\end{proof}

\begin{remark}
It is noteworthy that the equivalence between (i) and (ii) has so far been
been established only in the context of order convergence and norm
convergence. Regarding un-convergence and uaw-convergence, our results not
only improve the existing findings but also provide a notably concise and
straighforward proof (see, for comparison, the proof of \cite[Theorem
3.1]{L-171}).
\end{remark}

One can expect an improvement of Lemma \ref{L1} in the presence of a $\theta
$-unit by examining only the convergence of $x_{a}\wedge e$ to $x\wedge e.$
This is completely false even in $\mathbb{R}$ by considering constant nets.
Nevertheless, we have the following refinement.

\begin{lemma}
Let $X$ be vector lattice equipped with a convergence $\theta$ satisfying
property $\left(  \ast\right)  $ and assume that $e$ is a $\theta$-unit. Then
the following conditions are equivalent for a positive net in $X.$

\begin{enumerate}
\item[(i)] $\left(  x_{\alpha}\right)  _{\alpha\in A}$ is $u\theta$-convergent
to $x;$

\item[(ii)] $x_{\alpha}\wedge ke\overset{\theta}{\longrightarrow}x\wedge ke$
for all $k\in\mathbb{N}.$
\end{enumerate}
\end{lemma}

\begin{proof}
We only need to prove the implication (ii) $\Longrightarrow$ (i). Let $u\in
X_{+}$ and consider that inequality%
\begin{align*}
\left\vert x_{\alpha}\wedge u-x\wedge u\right\vert  &  \leq\left\vert
x_{\alpha}\wedge u-x_{\alpha}\wedge u\wedge k\right\vert \\
&  +\left\vert x_{\alpha}\wedge u\wedge k-x\wedge u\wedge k\right\vert
+\left\vert x\wedge u-x\wedge u\wedge k\right\vert .
\end{align*}
Using Birkhoff Inequality yields%
\[
\left\vert x_{\alpha}\wedge u-x\wedge u\right\vert \leq2\left\vert u-u\wedge
ke\right\vert +\left\vert x_{\alpha}\wedge ke-x\wedge ke\right\vert .
\]
It follows from (ii) that $\left\vert x_{\alpha}\wedge ke-x\wedge
ke\right\vert \overset{\theta}{\longrightarrow}0$ and from the fact that $e$
is a $\theta$-unit that $\left\vert u-u\wedge ke\right\vert \overset{\theta
}{\longrightarrow}0.$ Thus $\left\vert x_{\alpha}\wedge u-x\wedge u\right\vert
\overset{\theta}{\longrightarrow}0$ and then $x_{\alpha}\overset
{\mathfrak{u\theta}}{\longrightarrow}x$ by Lemma \ref{L1}.
\end{proof}

Assume that $Y$ is a sublattice of $X$ and $X$ is equipped with a convergence
$\theta.$ Then $\theta$ induces a natural convergence on $Y.$ However, when
considering $u\theta$-convergence in $Y,$ two distinct convergences emerge:
one induced by $X,$ and the other induced by $Y,$ as outlined in Definition
\ref{Y4-A}, where $b$ and $c$ belong to $Y$. it is important to note that
$u\theta$-convergence within$Y$ does not necessarily imply $u\theta
$-convergence within $X$ in general. Neverthekless, there exist signification
instances when this implication holds true. The following result extends
\cite[Theorem 4.3]{L-171}, \cite[Theorem 3]{L-358}, and \cite[Lemma 3]{L-310}
and provides a unified proof.

\begin{proposition}
\label{UP}\textit{Let }$X$ be a vector lattice equipped with a convergence
$\theta$ and let $Y$ \textit{be a sublattice of }$X\mathit{.}$ We assume that
$\theta$ is a convergence satisfying property $\left(  \ast\right)  .$
\textit{Let }$\left(  y_{\alpha}\right)  $ be \textit{a net in} $Y$
\textit{such that} $y_{\alpha}\overset{u\theta}{\longrightarrow}0$ \textit{in
}$Y$\textit{. Each of the following conditions implies that} $y_{\alpha
}\overset{u\theta}{\longrightarrow}0$ \textit{in} $X$.

\begin{enumerate}
\item[\textbf{(i)}] $Y$ \textit{is majorizing in} $X$;

\item[\textbf{(ii)}] $Y$ \textit{is} $\theta$-\textit{dense in} $X$;

\item[\textbf{(iii)}] $Y$ \textit{is a projection band in} $X$.
\end{enumerate}
\end{proposition}

\begin{proof}
We can assume without loss of generality that $y_{\alpha}\geq0$ for every
$\alpha$. (i) is trivial. To prove (ii), take $u\in X_{+}$ and let $\left(
z_{\beta}\right)  _{\beta\in B}$ be a net in $Y_{+}$ such that $z_{\beta
}\overset{\theta}{\longrightarrow}u.$ It follows from Birkhoff Inequality that%
\[
y_{\alpha}\wedge u=y_{\alpha}\wedge u-y_{\alpha}\wedge z_{\beta}+y_{\alpha
}\wedge z_{\beta}\leq\left\vert u-z_{\beta}\right\vert +y_{\alpha}\wedge
z_{\beta}.
\]
Now since $\theta$ satisfies Property $\left(  \ast\right)  $ it follows that
$y_{\alpha}\wedge u\overset{\theta}{\longrightarrow}0.$

(iii) Every $u\in X_{+}$ has the decomposition $u=u^{Y}+u^{Y^{d}}$ for some
$u^{Y}\in Y_{+}$ and $u^{Y^{d}}\in Y_{+}^{d}$. It follows from \cite[Theorem
6.5]{b-1089} that $y_{\alpha}\wedge u=y_{\alpha}\wedge v$ and then $y_{\alpha
}\wedge u\overset{\theta}{\longrightarrow}0.$
\end{proof}

\begin{remark}
For $uo$-convergence, the picture is clear thank to \cite[Theorem 3.2]{L-65}
which states that a sublattice $Y$ is regular if and only if uo-null nets in
$Y$ are uo-null in $X.$ For un-convergence there is no such characterization
and un-null nets in $Y$ may fail to be un-null in $X$ even if $Y$ is a band
(see \cite[Example 3.4]{L-171}).
\end{remark}

\begin{remark}

\begin{enumerate}
\item[(i)] It follows from Proposition \ref{UP} that if $X$ is a KB-space then
every $u\theta$-null net in $X$ is $u\theta$-null in $X^{\ast\ast}$ as $X$ is
a projection band in $X^{\ast\ast};$ see \cite[Theorem 4.60]{b-240}.

\item[(ii)] It can be shown as in \cite[Proposition 4.8]{L-171} that any band
in $X$ is $u\theta$-closed.
\end{enumerate}
\end{remark}

If $\left(  X_{i},\theta_{i}\right)  $ is a family of Riesz spaces $X$ is
equipped with product convergence $\theta$ defined as follows : $x_{\alpha
}\overset{\theta}{\longrightarrow}x$ if and only if $p_{i}\left(  x_{\alpha
}\right)  \overset{\theta_{i}}{\longrightarrow}p_{i}\left(  x\right)  $ for
each $i\in I.$ It is easily seen then that the inbounded convergence $u\theta$
satisfies%
\[
x_{\alpha}\overset{u\theta}{\longrightarrow}x\text{ }\Longleftrightarrow
p_{u}\left(  x_{\alpha}\right)  \overset{\theta_{i}}{\longrightarrow}%
p_{i}\left(  x\right)  \text{ for each }i\in I.
\]

We simply write this as $u\theta=%
{\textstyle\prod\limits_{i}}
u\theta_{i}.$ This is wat said Theorem 3.1 in \cite{L-310}. In this direction
we suggest another proof using convergence of nets. Let $\left(  x_{\alpha
}\right)  =\left(  x_{\alpha}\left(  i\right)  \right)  $ be a net in $X$ then
$x_{\alpha}\overset{u\tau}{\longrightarrow}x$ if and only if for all
$u=\left(  u\left(  i\right)  \right)  \in X_{+}$ we have $\left\vert
x_{\alpha}-x\right\vert \wedge u\overset{\tau}{\longrightarrow}0$ if and only
if $\left\vert x_{\alpha}\left(  i\right)  -x\left(  i\right)  \right\vert
\wedge u\left(  i\right)  \overset{\tau_{i}}{\longrightarrow}0$ for all $i$ if
and only if $x_{\alpha}\left(  i\right)  \overset{u\tau_{i}}{\longrightarrow
}x\left(  i\right)  $ for all $i$ if and only if $x_{\alpha}\longrightarrow x$
in $\left(  X,%
{\textstyle\prod\limits_{i}}
u\tau_{i}\right)  .$

It was shown in \cite[Proposition 2.12.]{L-310} than if $X$ is a Hausdorff
locally solid vector lattice with a Lebesgue topology $\tau$ and $Y$ is a
vector sublattice of \ $X,$ then $Y$ is $\tau$-closed if and only if $Y$ is
$u\tau$-closed. This result is a generalization of an earlier result shown in
\cite{L-65} and the method of the proof is almost the same. We show next that
the conclusion remains true without assuming that the spece has the Lebesgue
property. This answers positively Question 2.13 in \cite{L-310}. Notice also
that the proof presented here is remarkably simpler. This can be deduced from
more general result.

\begin{proposition}
\label{T12}Let $X$ be a vector lattice equipped with a convergence $\theta$
and let $Y$ be a vector lattice of $X.$ Then $Y$ is $\theta$-closed if and
only if $Y$ is $u\theta$-closed.
\end{proposition}

\begin{proof}
It is clear that $u\theta$-closedness implies $\theta$-closedness. Assume now
that $Y$ is $\theta$-closed and let $\left(  x_{\alpha}\right)  $ be a net in
$Y$ which is $u\theta$-convergent to $x.$ By considering the nets $\left(
x_{\alpha}^{\pm}\right)  $ we may assume that $x_{\alpha}\geq0$ for all
$\alpha.$ By Lemma \ref{L1}, $x_{\alpha}\wedge u\overset{\theta}%
{\longrightarrow}x\wedge u$ for all $u\in X_{+}.$ In particular $x_{\alpha
}\wedge x_{\beta}\overset{\theta}{\longrightarrow}x\wedge x_{\beta}$ for all
$\beta\in A.$ This shows that $x\wedge x_{\alpha}\in Y$ for all $\beta\in A.$
Now since $x_{\alpha}\wedge x\overset{\theta}{\longrightarrow}x\wedge x=x$ we
get $x\in Y,$ as required.
\end{proof}

A sequential version of Proposition \ref{T12} can be obtained in a similar
way. This answers a question asked by Taylor in \cite{L-310}.

\begin{proposition}
Let $X$ be a vector lattice and $Y$ a Riesz subspace of $X.$ Then $Y$ is
sequentially $\theta$-closed if and only if $Y$ is sequentially $u\theta$-closed.
\end{proposition}

\section{Strong completeness}

We fix now a conditional Riesz triple $\left(  E,e,T\right)  $ as it is
defined above. The main objective in this section is to show that the space
$L^{p}\left(  T\right)  $ is strongly complete for all $p\in\left[
1,\infty\right]  .$ To the best of our knowledge, such completeness has only
been established for tow specific values of $p.$ The case $p=\infty$ is no so
hard and was treated in \cite{L-360}, whiereas the case $p=2$ is significantly
more complex and was recently proved in \cite{L-886}. It is noteworthy that
the issue of strong sequential completeness has already been adressed. The
initial result, presented in \cite{L-360}, establishes that the space
$L^{1}\left(  T\right)  $ is strongly sequentially complete. The author, in
\cite{L-1111}, subsequently resolved the problem for the remaining values of
$p.$ In this paper we presente a comprehensive solution of the problem of
strong completeness. Indeed our main result states that proving strong
completeness for a single value of $p\in\lbrack1,\infty)$ is sufficent to
extend the result for the other values. It should be noted that for the
$T$-strong convergence in $L^{1}\left(  T\right)  $ the corresponding
unbounded convergence is the convergence in $T$-conditional probability. This
extension mirrors the behavior observed in classical $L^{1}\left(
\mathbb{P}\right)  $ spaces.

The main result of this section is the following.

\begin{theorem}
\label{Main}Let $\left(  E,e,T\right)  $ be a conditional Riesz triple and
$p\in\left(  1,\infty\right)  .$ Then the following statements are equivalent.

\begin{enumerate}
\item[(i)] $L^{1}\left(  T\right)  $ is strongly complete;

\item[(ii)] $L^{p}\left(  T\right)  $ is strongly complete.
\end{enumerate}
\end{theorem}

\begin{proof}
(i) $\Longrightarrow$ (ii) Assume that $L^{1}\left(  T\right)  $ is strongly
complete and let $\left(  x_{\alpha}\right)  _{\alpha\in A}$ be a net in
$L^{p}\left(  T\right)  $ which is strongly Cauchy, that is the net $\left(
\left\Vert x_{\alpha}-x_{\beta}\right\Vert _{T,p}\right)  _{\left(
\alpha,\beta\right)  \in A\times A}$ converges in order to $0.$ Then there
exists a net $\left(  u_{\gamma}\right)  _{\gamma\in\Gamma}$ in $L^{p}\left(
T\right)  $ such that $u_{\gamma}\downarrow0$ and for each $\gamma\in\Gamma$
there exists $\alpha_{\gamma}\in A$ such that $\left\Vert x_{\alpha}-x_{\beta
}\right\Vert _{T,p}\leq u_{\gamma}$ for every $\alpha,\beta\geq\alpha_{\gamma
}.$ By considering a tail of the net $\left(  x_{\alpha}\right)  $ we can
assume without lost of generality that $\left(  x_{\alpha}\right)  _{\alpha\in
A}$ is norm bounded in $L^{p}\left(  T\right)  ,$ that is $\sup\limits_{\alpha
}T\left(  \left\vert x_{\alpha}\right\vert ^{p}\right)  =M\in L^{1}\left(
T\right)  .$ On the other hand it follows from Lyapunov Inequality
\cite[Corollary 3.8]{L-180} that $\left(  x_{\alpha}\right)  _{\alpha\in A}$
is strongly Cauchy in $L^{1}\left(  T\right)  $ (with respect to the norm
$\left\Vert .\right\Vert _{T,1}$). It follows from (i) tthat $T\left\vert
x_{\alpha}-x\right\vert \overset{o}{\longrightarrow}0$ for some $x\in
L^{1}\left(  T\right)  .$ This shows, in particular, that $\left(  x_{\alpha
}\right)  _{\alpha\in A}$ converges to $x$ in $T$-conditional probability. By
considering the positive and negative parts of $x_{\alpha}$ we can assume that
$x_{\alpha}\geq0$ for all $\alpha\in A.$ We claim that $T\left(  x_{\alpha
}^{p}\wedge ke\right)  \overset{o}{\longrightarrow}T\left(  x^{p}\wedge
ke\right)  $ for all positive real $k.$ To this end we will use the following
inequality, which is true for all $a,b\in X_{+},$%
\[
\left\vert a^{p}-b^{p}\right\vert \leq p\left\vert a-b\right\vert \left(
a\vee b\right)  ^{p-1}.
\]
This yields the following%
\[
T\left\vert x_{\alpha}^{p}\wedge ke-x^{p}\wedge ke\right\vert \leq
pk^{1-1/p}T\left\vert x_{\alpha}-x\right\vert ,
\]
and then $T\left(  x_{\alpha}^{p}\wedge ke\right)  \overset{o}{\longrightarrow
}T\left(  x^{p}\wedge ke\right)  .$ We deduce from this that $T\left(
x^{p}\wedge ke\right)  \leq M$ for every $k.$ By taking the supremum over $k$
we get $Tx^{p}\leq M,$ and, in particular, $x\in L^{p}\left(  T\right)  .$

With this understanding, we now proceed now to establish that $\left\vert
x_{\alpha}-x\right\vert ^{p}$ converges to $0$ in $T$-conditional probability.
According to Lemma \ref{UI-B} it is sufficient to show that $T\left\vert
x_{\alpha}-x\right\vert ^{p}\wedge e\overset{o}{\longrightarrow}0.$ But this
follows from what we have proved above as the net $\left(  x_{\alpha
}-x\right)  $ is $\left\Vert .\right\Vert _{T,p}$-Cauchy in $L^{p}\left(
T\right)  $ and converges with respect to $\left\Vert .\right\Vert _{1,T}$ to
$0$. In particular we have
\begin{equation}
T\left\vert x_{\alpha}-x\right\vert ^{p}\wedge u\overset{o}{\longrightarrow
}0,\text{ for all }u\in L^{1}\left(  T\right)  _{+}. \label{Pr1}%
\end{equation}
We are now ready to show that $\left(  x_{\alpha}\right)  $ converges strongly
to $x$ in $L^{p}\left(  T\right)  .$ Let $\gamma\in\Gamma$ and let $\alpha
\geq\gamma.$ Then%
\begin{align*}
T\left\vert x_{\alpha}-x\right\vert ^{p}  &  =T\left(  \left\vert x_{\alpha
}-x\right\vert ^{p}\wedge\left\vert x_{\gamma}-x\right\vert ^{p}\right)
+T\left(  \left\vert x_{\alpha}-x\right\vert ^{p}-\left\vert x_{\gamma
}-x\right\vert ^{p}\right)  ^{+}\\
&  \leq T\left(  \left\vert x_{\alpha}-x\right\vert ^{p}\wedge\left\vert
x_{\gamma}-x\right\vert ^{p}\right)  +T\left\vert \left\vert x_{\alpha
}-x\right\vert ^{p}-\left\vert x_{\gamma}-x\right\vert ^{p}\right\vert .
\end{align*}
On one hand it follows from (\ref{Pr1}) that%
\[
T\left(  \left\vert x_{\alpha}-x\right\vert ^{p}\wedge\left\vert x_{\gamma
}-x\right\vert ^{p}\right)  \overset{o}{\longrightarrow}0.
\]
On the other hand we have%
\begin{align*}
T\left\vert \left\vert x_{\alpha}-x\right\vert ^{p}-\left\vert x_{\gamma
}-x\right\vert ^{p}\right\vert  &  \leq pT\left(  \left\vert \left\vert
x_{\alpha}-x\right\vert -\left\vert x_{\gamma}-x\right\vert \right\vert
\left(  \left\vert x_{\alpha}-x\right\vert \vee\left\vert x_{\beta
}-x\right\vert \right)  ^{p-1}\right) \\
&  \leq pT\left(  \left\vert x_{\alpha}-x\right\vert \left(  \left\vert
x_{\alpha}-x\right\vert +\left\vert x_{\gamma}-x\right\vert \right)
^{p-1}\right) \\
&  \leq p\left\Vert x_{\alpha}-x_{\beta}\right\Vert _{T,p}\left\Vert \left(
x_{\alpha}\vee x_{\gamma}\right)  ^{p-1}\right\Vert _{T,q}\\
&  =p\left\Vert x_{\alpha}-x_{\beta}\right\Vert _{T,p}T\left(  \left(
x_{\alpha}\vee x_{\gamma}\right)  ^{p}\right)  ^{1/q}\\
&  \leq p2^{p/q}M^{1/q}u_{\gamma}\leq Cu_{\gamma}.
\end{align*}
This shows that
\[
\limsup\limits_{\alpha}T\left(  \left\vert x_{\alpha}-x\right\vert
^{p}\right)  \leq Cu_{\gamma}%
\]
for all $\gamma\in\Gamma.$ Thus $\limsup T\left(  \left\vert x_{\alpha
}-x\right\vert ^{p}\right)  =0$, which shows that $\left(  x_{\alpha}\right)
_{\alpha\in A}$ converges strongly to $x$ in $L^{p}\left(  T\right)  .$ Thus
$L^{p}\left(  T\right)  $ is strongly complete as required.

(ii) $\Longrightarrow$ (i) Conversely assume now that $L^{p}\left(  T\right)
$ is strongly complete for some $p\in\left(  1,\infty\right)  $ and let
$\left(  x_{\alpha}\right)  _{\alpha\in A}$ be a strongly Cauchy net in
$L^{1}\left(  T\right)  .$ As above by considering a tail of the net $\left(
x_{\alpha}\right)  $ we can assume that $\left(  x_{\alpha}\right)  $ is
$T$-bounded, that is, there exists $M\in L^{1}\left(  T\right)  $ such that%
\[
\left\vert Tx_{\alpha}\right\vert \leq M\qquad\text{for all }\alpha\in A.
\]
Also we may assume that $x_{\alpha}\geq0.$ First it can be shown easily that
the net $\left(  x_{\alpha}\wedge ke\right)  _{\alpha\in A}$ is a $\left\Vert
.\right\Vert _{T,p}$-Cauchy net in $L^{p}\left(  T\right)  $ for every integer
$k\in\mathbb{N}.$ It follows from our assumption that $\left(  x_{\alpha
}\wedge ke\right)  _{\alpha\in A}$ converges to $z_{k}$ strongly in
$L^{p}\left(  T\right)  ,$ which implies, in particular, that $x_{\alpha
}\wedge ke\longrightarrow z_{k}$ strongly in $L^{1}\left(  T\right)  .$ We
derive from this that $Tz_{k}\leq M$ for all $k\in\mathbb{N}.$ It is also
clear that the net $\left(  z_{k}\right)  _{k\geq1}$ is increasing. Let us
denote its supremum in $X^{s}$ by $z.$ We will show that $z\in L^{1}\left(
T\right)  $ and that $x_{\alpha}\longrightarrow z$ in $T$-conditional
probability. Observe that%
\[
\left(  x_{\alpha}\wedge je\right)  \wedge ke=\left(  x_{\alpha}\wedge
je\right)  \wedge ke,
\]
and then%
\[
z_{k}\wedge je=ke\wedge z_{j}.
\]
Taking the supremum over $k$ we obtain%
\[
z\wedge je=z_{j}.
\]
According to Lemma \ref{L1} we deduce that $x_{\alpha}\longrightarrow x$ in
$T$-conditional probability. We show in this last step that $x_{\alpha
}\longrightarrow x$ strongly. The proof is very similar to the first part. By
our assumption, we have%
\[
u_{\gamma}:=\sup\limits_{\alpha,\beta\geq\gamma}T\left\vert x_{\alpha
}-x_{\beta}\right\vert \downarrow0.
\]
Observe now that for $\alpha\geq\beta$ we have%
\begin{align*}
T\left\vert x_{\alpha}-x\right\vert  &  =T\left\vert x_{\alpha}-x\right\vert
\wedge\left\vert x_{\beta}-x\right\vert +T\left(  x_{\alpha}-x_{\beta}\right)
^{+}\\
&  \leq T\left\vert x_{\alpha}-x\right\vert \wedge\left\vert x_{\beta
}-x\right\vert +T\left\vert x_{\alpha}-x_{\beta}\right\vert .\\
&  \leq Tx_{\alpha}\wedge\left\vert x_{\beta}-x\right\vert +u_{\beta}.
\end{align*}
Taking the $\lim\sup$ over $\alpha$ we get%
\[
\limsup\limits_{\alpha}Tx_{\alpha}\leq u_{\beta}.
\]
As this happens for every $\beta$ we obtain $\limsup Tx_{\alpha}=0,$ which
completes the proof.
\end{proof}

As a consequence of Theorem \ref{Main}, and \cite[Theorem 4.1]{L-360} we
deduce the following.

\begin{theorem}
\label{Y4-C}Let $\left(  E,e,T\right)  $ be a conditional Riesz triple. Then
$L^{p}\left(  T\right)  $ is strongly complete for every $p\in\left[
1,\infty\right]  .$
\end{theorem}

\begin{remark}
If our focus is solely on establishing the strong completeness of
$L^{1}\left(  T\right)  ,$ then only the second part of the proof of Theorem
\ref{Main} suffices. Given the intricate nature of proving the strong
completeness of $L^{2}\left(  T\right)  ,$ an accomplishment of papers
\cite{L-360} and \cite{L-928}, Theorem \ref{Main} opens the door to
envisioning a more straightforward proof for the strong completeness of
$L^{p}\left(  T\right)  $ for $p\in\left(  1,\infty\right)  $ as this can be
obtained by showing the result for just one value of $p.$
\end{remark}

\section{Completeness for convergence in T-conditional probability}

We consider again a conditional Riesz triple $\left(  X,e,T\right)  .$ Our
goal in this section is to prove the completeness of the space $X^{u}$ when
endowed with convergence in $T$-conditional probability.

\begin{theorem}
\label{Y4-B}Let $\left(  X,e,T\right)  $ be a conditional Riesz triple. Then
the space $X^{u}$ is complete with respect to convergence in $T$-conditional probability.
\end{theorem}

\begin{proof}
Consider a Cauchy net $(x_{\alpha})_{\alpha\in A}$ in $X^{u}$ with respect to
convergence in $T$-conditional probability. We have to show that $(x_{\alpha
})_{\alpha\in A}$ is convergent in $T$-conditional probability. By considering
the nets $(x_{\alpha}^{+})_{\alpha\in A}$ and $(x_{\alpha}^{-})_{\alpha\in A}$
we may assume that $(x_{\alpha})_{\alpha\in A}$ is positive. For each $u\in
X_{+}$ there exists $\ell_{u}\in X_{+}$ such that $T\left\vert x_{\alpha
}\wedge u-\ell_{u}\right\vert \overset{o}{\longrightarrow}0$. We put
$\ell=\sup\limits_{u\in X_{+}}\ell_{u}\in X^{s}$. We split the proof into two
steps. We will show first that $\ell_{u}=\ell\wedge u$ for every $u\in X_{+}$
and then we prove that $\ell\in X^{u}.$ Thus we will get $T\left\vert
x_{\alpha}\wedge u-\ell\wedge u\right\vert \overset{o}{\longrightarrow}0$ for
every $u\in X_{+}.$ This allows us to conclude via the folloing inequality,
which holds for all $x,y,u\in X_{+}^{u}:$%
\[
\left(  x-y\right)  \wedge u\leq2\left(  x\wedge u-y\wedge y\right)  ,
\]
which in turn yields%
\[
T\left(  \left\vert x_{a}-x\right\vert \wedge u\right)  \leq2T\left\vert
x_{\alpha}\wedge u-x\wedge u\right\vert .
\]

\textbf{Step 1.} We will show first that%
\[
\ell_{a}=\ell\wedge a\text{ for every}\ a\in X_{+}.
\]
We argue by contradiction and assume that $x_{\alpha}\wedge a\longrightarrow
\ell_{a}<\ell\wedge a$ strongly in $L^{1}\left(  T\right)  $ for some $a\in
X_{+}$. According to \cite[Property P6, pp. 217]{L-444} we know that%
\[
\ell\wedge a=\left(  \sup\limits_{u\in X_{+}}\ell_{u}\right)  \wedge
a=\sup\limits_{u\in X_{+}}(\ell_{u}\wedge a).
\]
\newline Hence there exists $b\in X_{+}$ such that%
\[
\ell_{b}\wedge a>\ell_{b}\wedge a\wedge\ell_{a}=\ell_{b}\wedge\ell_{a}.
\]
On one hand we have%
\[
x_{\alpha}\wedge b\wedge a=(x_{\alpha}\wedge b)\wedge a\overset{\rho
}{\longrightarrow}\ell_{b}\wedge a,
\]
\newline and on the other hand%
\[
x_{\alpha}\wedge b\wedge a=(x_{\alpha}\wedge b)\wedge(x_{\alpha}\wedge
a)\overset{\rho}{\longrightarrow}\ell_{b}\wedge\ell_{a},
\]
which leads to the contradiction $\ell_{a}\wedge\ell_{b}=\ell_{b}\wedge a.$

\textbf{Step 2.} We claim that $\ell\in X^{u}$. We argue again by
contradiction and we suppose that $\ell\in X^{s}\backslash X^{u}$. Then by
\cite[Corollary 15]{L-444} there exists $u\in X_{+}$ such that%
\[
\ell\geq nu>0.\text{ for all }n\in\mathbb{N}.
\]
\newline Now, put $L=\sup nu\in X^{s}$ and observe that $x_{\alpha}\wedge
nu\uparrow x_{\alpha}\wedge L$ as $n\longrightarrow\infty$ and that, for all
$y\in X_{+}$, we have%
\[
(x_{\alpha}\wedge L)\wedge y=(x_{\alpha}\wedge y)\wedge L\overset{\rho
}{\longrightarrow}(\ell\wedge y)\wedge L=y\wedge L.
\]
Hence, by considering the net $(x_{\alpha}\wedge L)$ instead of $(x_{\alpha})$
and $L$ instead of $\ell$, we may assume that $x_{\alpha}\leq\ell=\sup nu$. In
particular, for every $\alpha\in A$ and $n\in\mathbb{N}$, we have%
\begin{equation}
\left(  x_{\alpha}\wedge nu\right)  _{n}\overset{o}{\longrightarrow}x_{\alpha
}\text{ and }\left(  x_{\alpha}\wedge nu\right)  _{\alpha}\overset
{o}{\longrightarrow}nu. \tag{$\left(       \ast\right)       $}%
\end{equation}

For fixed $\alpha\in A$, we can find a net $(u_{\gamma})_{\gamma\in\Gamma}$
with $u_{\gamma}\downarrow0$ and such that for every $\gamma\in\Gamma$ there
is $n_{\gamma}$ satisfying:%
\[
|x_{\alpha}\wedge nu-x_{\alpha}|\leq u_{\gamma}\qquad n\geq n_{\gamma}.
\]
Pick $\gamma$ in $\Gamma$ and choose $n_{\gamma}$ as above. Then we have for
all $n\geq n_{\gamma}$,%
\begin{align*}
x_{\beta}-x_{\alpha}  &  =x_{\beta}-x_{\beta}\wedge(n+1)u\\
&  +x_{\beta}\wedge(n+1)u-x_{\alpha}\wedge nu+x_{\alpha}\wedge nu-x_{\alpha}\\
&  \geq x_{\beta}\wedge(n+1)u-x_{\alpha}\wedge nu-|x_{\alpha}\wedge
nu-x_{\alpha}|\\
&  \geq x_{\beta}\wedge(n+1)u-nu-u_{\gamma}.
\end{align*}
It follows that%
\begin{align*}
T|x_{\beta}-x_{\beta^{\prime}}|\wedge u  &  \geq\lim_{\beta}T\left(
(x_{\beta}\wedge(n+1)u-nu-u_{\gamma})\wedge u\right) \\
&  =T\left(  (u-u_{\gamma})\wedge u\right)  ,
\end{align*}
where the last equality follows from $(\ast)$. Since this holds for every
$\gamma$ in $\Gamma$ we conclude that%
\[
\sup\limits_{\beta,\beta\geq\alpha}T|x_{\beta}-x_{\beta^{\prime}}|\wedge u\geq
Tu.
\]
\newline But this occurs for every $\alpha\in A$, which contradicts the fact
that the net $(x_{\alpha})$ is Cauchy with respect with $T$-conditional
probability and completes the proof.
\end{proof}

Thanks to Theorem \ref{Y4-B} we obtain a short proof of the hard part on the
main result in \cite{L-444}.

\begin{corollary}
Every universally complete vector lattice is unbounded order complete.
\end{corollary}

\begin{proof}
Consider a universally complete Riesz space $X.$ We choose a weak unit $e$ for
$X$ and we consider the conditional expectation operator $T=\operatorname*{Id}%
\nolimits_{X}.$ Then the convergence in $T$-conditional probability coincides
with the unbounded order convergence. Thus Theorem \ref{Y4-B} tells us that
$X=X^{u}$ is unbounded order complete.
\end{proof}

\begin{remark}
In the special case when $E=L_{0}\left(  \mu\right)  $ where $\mu$ is a
probability measure and $T$ is the expectation operator, Theorem Y4-B yields
that t\textit{he topology of convergence in probability defines a complete
metrizable locally solid topology on} $L_{0}(\mu)$ (\textit{see \cite[Theorem
13.41]{b-187})}.
\end{remark}

\section{Applications}

Consider a conditional expectation preserving system represented by the
4-tuple $\left(  E,e,T,S\right)  ,$ where $\left(  E,e,T\right)  $ is a
conditional Riesz triple, and $S$ is an order continuous Riesz homomorphism on
$E$ that satisfies $Se=e$ and $TS=T.$ It was shown in \cite[Theorem 3.7]{L-33}
that if $E$ is $T$-universally complete (as in the case when $E=L^{1}\left(
T\right)  $), then for any $x\in L^{1}\left(  T\right)  ,$ the sequence
$S_{n}x:=\dfrac{1}{n}\sum\limits_{k=0}^{n-1}S^{k}x$ is order convergent to
$L_{S}x,$ where $L_{S}$ is a conditional expectation operator. Furthermore,
the system is ergodic if and only if $L=T.$ More recently in \cite{L-1100}, it
was additionally shown that $S$ acts as an isometry from $L^{p}\left(
T\right)  $ to $L^{p}\left(  T\right)  .$ With these considerations and
relying on Theorem \ref{Y4-C}, we are able to extend some results from the
classical measure theory presented in \cite[Theorem 8.8]{b-2982} to the
setting of Riesz spaces. Let us recall that if $\left(  \Omega,\mathcal{F}%
,\mu,\tau\right)  $ is a measure preserving system and $1\leq p<\infty$ and
$S$ is the operator defined as $Sf=f\circ\tau,$ for $f\in L^{1}\left(
\mu\right)  ;$ then the Ces\`{a}ro means $\dfrac{1}{n}\sum\limits_{k=0}%
^{n-1}S^{k}f$ converge in $L^{p}\left(  T\right)  $ for every $f\in
L^{p}\left(  T\right)  $ (see \cite[Theorem 8.8]{b-2982}). Moreover, the its
limit $P$ is a projection with its range forming a Banach sublattice of
$L^{p}\left(  \mu\right)  $ and $P_{\tau}$ satisfies $\mathbb{E}%
Pf=\mathbb{E}f$ for every $f\in L^{p}\left(  \mu\right)  ,$ where $\mathbb{E}$
designs the expectation operator. Further\`{u}ore, $P\left(  fPg\right)
=PfPg$ for $f\in L^{p},g\in L^{q}$ with $\dfrac{1}{p}+\dfrac{1}{q}=1.$ We will
extend these results to the setting of Riesz spaces. Before presenting our
results we need the following lemma.

\begin{lemma}
\label{Y4-D}Let $\left(  E,e,T\right)  $ be a conditional Riesz triple. Then
the ideal $E_{e}$ generated by $e$ is strongly dense in $L^{p}\left(
T\right)  .$ In particular the space $L^{\infty}\left(  T\right)  $ is
strongly dense in $L^{p}\left(  T\right)  $ for every $p\in\lbrack1,\infty).$
\end{lemma}

\begin{proof}
Let $x\in L^{p}\left(  T\right)  _{+}.$ Then $x\wedge ne\uparrow x$ and then
$\left\vert x\wedge ne-x\right\vert ^{p}\overset{o}{\longrightarrow}0.$ Hence
$\left\Vert x\wedge ne-x\right\Vert _{T,p}\longrightarrow0$, which proves the lemma.
\end{proof}

\begin{theorem}
\label{Y4-E}Let $\left(  E,e,T,S\right)  $ be a conditional expectation
preserving system. Then $\left(  S_{n}x\right)  $ convegres strongly in
$L^{p}\left(  T\right)  $ to $L_{S}x$ for every $x\in L^{p}\left(  T\right)
$. Moreover we have%
\begin{equation}
L_{S}\left(  xL_{S}y\right)  =L_{S}x.L_{S}y\text{ for all }x\in L^{p}\left(
T\right)  ,y\in L^{q}\left(  T\right)  \label{E1}%
\end{equation}
with $\dfrac{1}{p}+\dfrac{1}{q}=1.$
\end{theorem}

\begin{proof}
(i) We know by \cite[Theorem 3.9]{L-33} that $\left(  S_{n}x\right)  $ is
order convergent to $L_{S}x$ for every $x\in L^{1}\left(  T\right)  $ and then
strongly in norm $\left\Vert .\right\Vert _{1,T}.$ Thus if $\left(
S_{n}x\right)  $ is strongly convergent in $L^{p}\left(  T\right)  $ then it
necessarily converges to $L_{S}x.$ We will conclude by showing that the
sabspace of $L^{p}\left(  T\right)  $ defined by%
\[
V=\left\{  x\in L^{p}\left(  T\right)  :\left(  S_{n}x\right)  \text{
converges in }L^{p}\left(  T\right)  \right\}
\]
is $L^{p}\left(  T\right)  .$ To this end we will show that $V$ is strongly
closed and contains $L^{\infty}\left(  T\right)  $ and derive the result from
Lemma \ref{Y4-D} is contained in $V$ and that $V$ is strongly closed in
$L^{p}\left(  T\right)  $.

(a) Let us establish first that $L^{\infty}\left(  T\right)  \subseteq V.$
Consider an element $x$ in $L^{\infty}\left(  T\right)  ;$ then $\left\vert
x\right\vert \leq r$ for some $r\in R\left(  T\right)  .$ According to
\cite[Theorem 3.9]{L-33}, we knwo that $\left\vert S_{n}x-Lx\right\vert
\overset{o}{\longrightarrow}0$ in $L^{1}\left(  T\right)  .$ Consequently,
$\left\vert S_{n}x-Lx\right\vert ^{p}\overset{o}{\longrightarrow}0$ in
$E^{u}=L^{1}\left(  T\right)  ^{u}.$ As $L^{p}\left(  T\right)  $ is a
Dedekind complete Riesz subspace of $E^{u},$ to prove that $\left\Vert
S_{n}x-x\right\Vert _{T,p}\overset{o}{\longrightarrow}0$ it is sufficient to
prove that $\left(  \left\vert S_{n}x-Lx\right\vert ^{p}\right)  $ is order
bounded in $L^{1}\left(  T\right)  .$ As $ST=T$ (see \cite[Lemma 3]{L-1100})
and $S$ is a Riesz homomorphism we have%
\[
\left\vert S^{k}x\right\vert ^{p}=\left(  S^{k}\left\vert x\right\vert
\right)  ^{p}\leq\left(  S^{k}r\right)  ^{p}=r^{p}.
\]
It follows from the convexity of the map $u\longmapsto\left\vert u\right\vert
^{p}$ that $\left\vert S_{n}\left(  x\right)  \right\vert ^{p}\leq r^{p}.$
Since $S_{n}x\overset{o}{\longrightarrow}L_{S}x$ we derive that $\left\vert
Lx\right\vert ^{p}\leq r^{p}$ and thereby the the sequence is $\left(
\left\vert S_{n}x-L_{S}x\right\vert ^{p}\right)  $ is order bounded in
$L^{1}\left(  T\right)  $ as desired.

(b) We show now that $V$ is strongly closed in $L^{p}\left(  T\right)  .$
Assume that $\left(  x_{\alpha}\right)  _{\alpha\in A}$ is a net in $V$ that
is strongly convergent to some element $x$ in $L^{p}\left(  T\right)  .$ We
aim to show that $x\in V.$ According to Theorem \ref{Y4-C} it is sufficent to
show that the sequence $\left(  S_{n}x\right)  $ is strongly Cauchy in
$L^{p}\left(  T\right)  .$ Now observe that we have for $n,k\in\mathbb{N},$%
\begin{align*}
\left\Vert S_{n}x-S_{k}x\right\Vert _{T,p}  &  \leq\left\Vert S_{n}%
x-S_{n}x_{\alpha}\right\Vert _{T,p}+\left\Vert S_{n}x_{\alpha}-S_{k}x_{\alpha
}\right\Vert _{T,p}+\left\Vert S_{k}x_{\alpha}-S_{k}x\right\Vert _{T,p}\\
&  \leq\left\Vert x-x_{\alpha}\right\Vert _{T,p}+\left\Vert S_{n}x_{\alpha
}-S_{k}x_{\alpha}\right\Vert _{T,p}+\left\Vert x_{\alpha}-x\right\Vert
_{T,p}\\
&  =2\left\Vert x-x_{\alpha}\right\Vert _{T,p}+\left\Vert S_{n}x_{\alpha
}-S_{k}x_{\alpha}\right\Vert _{T,p}.
\end{align*}
Thus%
\[
\limsup\limits_{n,k}\left\Vert S_{n}x-S_{k}x\right\Vert _{T,p}\leq2\left\Vert
x-x_{\alpha}\right\Vert _{T,p}.
\]
As this happens for every $\alpha$ we deduce that
\[
\limsup\limits_{n,k}\left\Vert S_{n}x-S_{k}x\right\Vert _{T,p}=0,
\]
and then $\left(  S_{n}x\right)  $ converges in $L^{p}\left(  T\right)  ,$ as
required. It remains to prove the identity (\ref{E1}). First, note that if
$x\in L^{p}\left(  T\right)  $ and $y\in L^{q}\left(  T\right)  ,$ the
previous result ensures $L_{S}y\in L^{q}\left(  T\right)  $, and the
application of H\"{o}lder Inequality (see \cite[Theorem 3.7]{L-180}) confirms
that $xL_{S}y\in L^{1}\left(  T\right)  $. Furthermore, as $L_{S}$ is
conditional expectation and is thus an averaging operator (\cite{L-24}), we
have
\[
L_{S}\left(  xL_{S}y\right)  =L_{S}xL_{S}y.
\]
This concludes the proof of the theorem.
\end{proof}

The final result of the paper provides a characterization of ergodicity.

\begin{theorem}
Let $\left(  E,e,T,S\right)  $ be a conditional expectation preserving system
and $p,q\in\left[  1,\infty\right]  $ with $\dfrac{1}{p}+\dfrac{1}{q}=1.$ Then
the following statements are equivalent.

\begin{enumerate}
\item[(i)] The system $\left(  E,e,T,S\right)  $ is ergodic;

\item[(ii)] The sequence $\dfrac{1}{n}%
{\textstyle\sum\limits_{k=0}^{n-1}}
T\left(  fS^{k}g\right)  $ is order convergent to $TfTg$ for all $f\in
L^{p}\left(  T\right)  ,g\in L^{q}\left(  T\right)  .$
\end{enumerate}
\end{theorem}

\begin{proof}
(i) $\Longrightarrow$ (ii) Assume the system is ergodic and let $f\in
L^{p}\left(  T\right)  ,g\in L^{q}\left(  T\right)  .$ In view of
\cite[Theorem 10]{L-1100} and Theorem \ref{Y4-E} we have $\left\Vert
S_{n}g-Tg\right\Vert _{T,p}\overset{o}{\longrightarrow}0.$ Now using
H\"{o}lder Inequality we get%
\begin{align*}
\left\vert T\left(  fS_{n}g\right)  -TfTg\right\vert  &  =T\left\vert f\left(
S_{n}g-Tg\right)  \right\vert \leq T\left\vert f\left(  S_{n}g-Tg\right)
\right\vert \\
&  \leq\left\Vert f\right\Vert _{p}\left\Vert S_{n}g-Tg\right\Vert _{q}.
\end{align*}
Hence $\dfrac{1}{n}%
{\textstyle\sum\limits_{k=0}^{n-1}}
T\left(  fS^{k}g\right)  \overset{o}{\longrightarrow}TfTg$ as
$n\longrightarrow\infty.$

(ii) $\Longrightarrow$ (i) This follows immediately from \cite[Theorem
12]{L-1100}.
\end{proof}

\end{document}